\documentclass[10pt]{amsart}
\usepackage{amsmath,amssymb,verbatim}
\usepackage[latin1]{inputenc}
\usepackage{enumerate}
\usepackage{pdflscape}
\usepackage{amsthm}
\usepackage{mathrsfs}
\usepackage{color}
\usepackage[normalem]{ulem}
\usepackage{tikz}
\usetikzlibrary{matrix}
\usepackage[all]{xy}
\usepackage{mathtools}

\newtheorem{theorem}{Theorem}[section]

\newtheorem{proposition}[theorem]{Proposition}

\newtheorem{corollary}[theorem]{Corollary}

\newcommand{\Z}{{\mathbb Z}}

\newcommand{\F}{{\mathbb F}}

\newcommand{\G}{\mathcal{G}}
\newcommand{\B}{\mathcal{B}}

\newcommand{\GEN}[1]{\left\langle #1 \right\rangle}

\newenvironment{proofof}{\par\noindent \textit{Proof of }}{\qed\par\bigskip}

\newcommand{\qand}{\quad \text{and} \quad}

\DeclareMathOperator{\Exp}{Exp}

\title[The Modular Isomorphism Problem for 2-generated groups of class 2]{The Modular Isomorphism Problem for two generated groups of class two}

\author{Osnel Broche and \'Angel del R\'\i o}
\thanks{The first author has been partially supported by Fundación Séneca of Murcia under a Jiménez de la Espada grant 20598/IV/18.
The second author has been partially supported by the Spanish Government under Grant MTM2016-77445-P with "Fondos FEDER" and, by Fundaci\'on S\'eneca of Murcia under Grants 19880/GERM/15}

\address{O. Broche: Departamento de Ci\^{e}ncias Exatas, Universidade Federal de Lavras, Caixa Postal 3037, 37200-000, Lavras, Brazil. \rm{osnel@ufla.br}}

\address{\'{A}. del R\'{i}o: Departamento de Matem\'{a}ticas, Universidad de Murcia, 30100, Murcia, Spain. \rm{adelrio@um.es}}



\begin{document}
\maketitle

\begin{abstract}
We prove that if $G$ is finite 2-generated $p$-group of nilpotence class at most 2 then the group algebra of $G$ with coefficients in the field with $p$ elements determines $G$ up to isomorphisms.
\end{abstract}

Let $G$ be a finite $p$-group and let $\F_p$ be the field with $p$ elements. 
A long standing question is whether the group algebra $\F_pG$ determines $G$ up to isomorphisms.
Formally it is the following problem:

\begin{quote}
\textbf{Modular Isomorphism Problem (MIP)}.\\
Does $\F_pG\cong \F_pH$ implies $G\cong H$ for $G$ and $H$ finite $p$-groups?
\end{quote}

Positive solutions for (MIP) has only been proved under strong conditions. For example, it has been verified for groups of order $p^n$ with $n\le 5$ \cite{Passman1965, Makasikis, Sandling1984, Sandling1989}, and, using computers it has been verified for groups of order $2^n$ with $n\le 9$ and $3^6$ \cite{Wursthorn1990, Wursthorn1993, BKRW,Eick2008,EickKonovalov2011} (see also the computer free treatement for groups of order $2^6$ in \cite{HertweckSoriano2006}). 
The (MIP) has also been proved for abelian \cite{Deskins1956,Coleman1964} and metacyclic groups \cite{Baginski1988, Sandling1996}.
Other results on the Modular Isomorphism Problem can be found in the introductory lists in \cite{BaginskiKonovalov} and \cite{HertweckSoriano2006}.

The Modular Isomorphism Problem is still open for groups of class 2 except in very restricted situations. For example, it has been proven for these groups in case the commutator is elementary abelian \cite[Theorem~6.25]{Sandling1984} or the center has index $p^2$ \cite{Drensky1989}. At least it is known that if $\F_p G\cong \F_p H$ and $G$ is a $p$-group of class 2 then so is $H$ \cite{BaginskiKonovalov} and if $G$ is $n$-generated then so is $H$. The aim of this paper is to prove (MIP) for 2-generated groups of class 2. Formally we prove the following result:

\begin{theorem}\label{Main}
Let $G$ be and $H$ be finite $p$-groups such that $\F_p G\cong \F_p H$. If $G$ is 2-generated group of class at most 2 then $G\cong H$. 
\end{theorem}

The basic idea consists in proving that the modular group algebra of a group satisfying the hypothesis of the theorem determines a 5-tuple of integers which in turn determines the group up to isomorphism. 
The 5-tuples belong to a certain set $\mathcal{A}$ so that there is a one-to-one correspondence between $\mathcal{A}$ and the isomorphism classes of 2-generated groups of class 2. 
That the set of isomorphism classes of such groups is in one-to-one with a certain set $\mathcal{A}'$ of 5-tuples of integers is already in the literature \cite{AMM}. However the set $\mathcal{A}'$ in this reference does not adapt to our purposes.  Figure~\ref{AdmissiblePoints} compares visually the sets $\mathcal{A}$, in this paper, and $\mathcal{A}'$, in \cite{AMM}. 
Well known results on the (MIP) readily imply that the first four entries in an element of $\mathcal{A}$ are determined by the group algebra. 
To prove the same for the fifth entry some additional arguments are needed using the Jenning series \cite{Jennings1941,Sehgal1978} and the ``kernel sizes'' \cite{Passman1965}.
Actually the Jenning series suffices to determine the fifth entry in case $p$ is odd. 
However, this does not hold for $p=2$. 

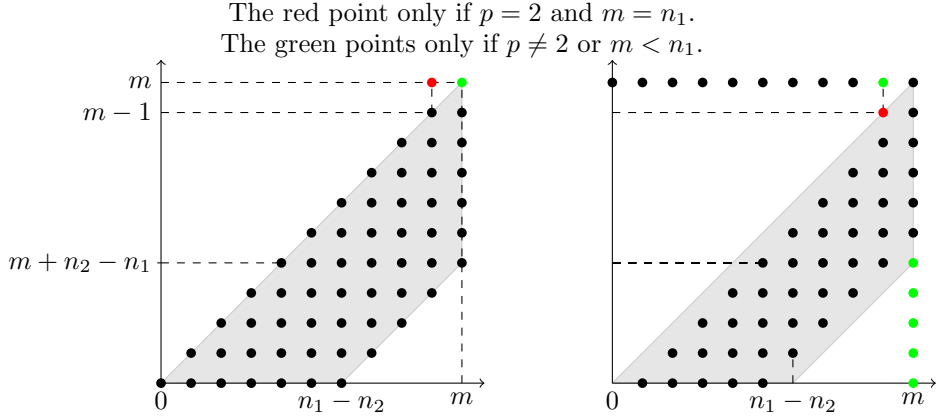
\begin{figure}[h!]
	\begin{center}
		\begin{tikzpicture}[scale=1]
		\filldraw[draw=black!20!white, fill=black!10!white] (0,0) -- (2.4,0) -- (4,1.6) -- (4,4) -- cycle;
		\draw[->] (0,0) -- (0,4.3);
		\draw[->] (0,0) -- (4.3,0);
		\draw (4,4.9) node {The red point only if $p=2$ and $m=n_1$.};
		\draw (4,4.5) node {The green points only if $p\ne 2$ or $m<n_1$.};
		\draw (0,0) node[below] {$0$};
		\draw (0,4) node[left] {$m$};
		\draw[dashed] (0,4) -- (4,4);
		\draw (4,0) node[below] {$m$};
		\draw[dashed] (4,0) -- (4,4);
		\draw (2.4,0) node[below] {$n_1-n_2$};
		\draw (-0.05,1.6) -- (0,1.6) node[left] {$m+n_2-n_1$};
		\draw[dashed] (0,1.6) -- (1.6,1.6);
		\draw (0,3.6) node[left] {$m-1$};
		\draw[dashed] (0,3.6) -- (3.6,3.6);
		\draw[dashed] (3.6,3.6) -- (3.6,4);	
		\filldraw[thick,red] (3.6,4) circle (.05cm);		
		\draw[dashed] (6,1.6) -- (8,1.6);
		\foreach \x in {0, 0.4, 0.8, 1.2, 1.6, 2, 2.4, 2.8, 3.2, 3.6}
		\filldraw[thick] (\x,\x) circle (.05cm);		
		\draw[dashed] (6,1.6) -- (8,1.6);
		\foreach \x in {0, 0.4, 0.8, 1.2, 1.6, 2, 2.4, 2.8, 3.2,3.6}
		\filldraw[thick] (\x+0.4,\x) circle (.05cm);		
		\draw[dashed] (6,1.6) -- (8,1.6);
		\foreach \x in {0, 0.4, 0.8, 1.2, 1.6, 2, 2.4, 2.8, 3.2}
		\filldraw[thick] (\x+0.8,\x) circle (.05cm);		
		\draw[dashed] (6,1.6) -- (8,1.6);	
		\foreach \x in {0, 0.4, 0.8, 1.2, 1.6, 2, 2.4, 2.8}
		\filldraw[thick] (\x+1.2,\x) circle (.05cm);		
		\draw[dashed] (6,1.6) -- (8,1.6);	
		\foreach \x in {0, 0.4, 0.8, 1.2, 1.6, 2, 2.4}
		\filldraw[thick] (\x+1.6,\x) circle (.05cm);		
		\draw[dashed] (6,1.6) -- (8,1.6);
		\foreach \x in {0, 0.4, 0.8, 1.2, 1.6, 2}
		\filldraw[thick] (\x+2,\x) circle (.05cm);		
		\draw[dashed] (6,1.6) -- (8,1.6);		
		\foreach \x in {0, 0.4, 0.8, 1.2, 1.6}
		\filldraw[thick] (\x+2.4,\x) circle (.05cm);		
		\draw[dashed] (6,1.6) -- (8,1.6);		
\filldraw[thick,green] (4,4) circle (.05cm);		
		\draw[dashed] (6,1.6) -- (8,1.6);		
		
		\filldraw[draw=black!20!white, fill=black!10!white] (6,0) -- (8.4,0) -- (10,1.6) -- (10,4) -- cycle;
		\draw[->] (6,0) -- (6,4.3);
		\draw[->] (6,0) -- (10.3,0);
		
		\draw (6,0) node[below] {$0$};
		\draw (10,0) node[below] {$m$};
		\draw (8.4,0) node[below] {$n_1-n_2$};
		
		\draw[dashed] (6,1.6) -- (8,1.6);
		\draw[dashed] (8.4,0) -- (8.4,0.4);
		\draw[dashed] (6,3.6) -- (9.6,3.6);
		\draw[dashed] (9.6,3.6) -- (9.6,4);	
		\filldraw[thick,red] (9.6,3.6) circle (.05cm);		
		\draw[dashed] (6,1.6) -- (8,1.6);
		\filldraw[thick] (10,4) circle (.05cm);		
		\draw[dashed] (6,1.6) -- (8,1.6);
		\filldraw[thick,green] (9.6,4) circle (.05cm);		
		\draw[dashed] (6,1.6) -- (8,1.6);
		\foreach \x in {0, 0.4, 0.8, 1.2, 1.6, 2, 2.4, 2.8, 3.2}
		\filldraw[thick] (6+\x,4) circle (.05cm);		
		\draw[dashed] (6,1.6) -- (8,1.6);
		\foreach \x in {0, 0.4, 0.8, 1.2, 1.6, 2, 2.4, 2.8, 3.2, 3.6}
		\filldraw[thick] (\x+6.4,\x) circle (.05cm);		
		\draw[dashed] (6,1.6) -- (8,1.6);
		\foreach \x in {0, 0.4, 0.8, 1.2, 1.6, 2, 2.4, 2.8, 3.2}
		\filldraw[thick] (\x+6.8,\x) circle (.05cm);		
		\draw[dashed] (6,1.6) -- (8,1.6);	
		\foreach \x in {0, 0.4, 0.8, 1.2, 1.6, 2, 2.4, 2.8}
		\filldraw[thick] (\x+7.2,\x) circle (.05cm);		
		\draw[dashed] (6,1.6) -- (8,1.6);	
		\foreach \x in {0, 0.4, 0.8, 1.2, 1.6, 2, 2.4}
		\filldraw[thick] (\x+7.6,\x) circle (.05cm);		
		\draw[dashed] (6,1.6) -- (8,1.6);
		\foreach \x in {0, 0.4, 0.8, 1.2, 1.6, 2}
		\filldraw[thick] (\x+8,\x) circle (.05cm);		
		\draw[dashed] (6,1.6) -- (8,1.6);		
		\foreach \x in {0, 0.4, 0.8, 1.2, 1.6}
		\filldraw[thick,green] (10,\x) circle (.05cm);		
		\draw[dashed] (6,1.6) -- (8,1.6);		
		\end{tikzpicture}
	\end{center}
	\caption{\label{AdmissiblePoints} The left picture represents the pairs $(s_1,s_2)$ satisfying the conditions (A1)-(A4) in Proposition~\ref{Clasificacion} for a given prime $p$ and positive integers $m\le n_2\le n_1$. (i.e. $(m,n_1,n_2,s_1,s_2)\in \mathcal{A}$).
	The right picture represents the corresponding pairs for the parameters in \cite{AMM} (i.e. $(m,n_1,n_2,s_1,s_2)\in \mathcal{A}'$).
		Actually, in these pictures $m=10$ and $n_1-n_2=6$.}
\end{figure}

We start introducing the basic notation which is mostly standard.
We denote by $C_n$ the cyclic group of order $n$.
The order, exponent and derived subgroup of a group are denoted $|G|$, $\Exp(G)$ and $G'$ respectively.
The order of a group element $g$ is denoted $|g|$ and if $h$ is another group element of the same group then $[g,h]=g^{-1}h^{-1}gh$.
We use bar notation in a group $G$ for reduction modulo $G'$.

If $p$ is a prime integer then $v_p$ denotes the $p$-adic valuation in $\Z$, i.e. for every positive integer $n$, $v_p(n)$ is the maximum non-negative integer $k$ with $p^k\mid n$.

We consider the groups given by the following presentation for non-negative integers $m,n_1,n_2,s_1,s_2$:
    $$G_{m,n_1,n_2}^{s_1,s_2}=
    \GEN{b_1,b_2 \mid [b_2,b_1]^{p^m}=[b_i,[b_2,b_1]]=1, b_i^{p^{n_i}}=[b_2,b_1]^{p^{s_i}}, (i=1,2)}.$$
Every finite 2-generated $p$-group of class 2 with $|G'|=p^m$ and $G/G'\cong C_{p^{n_1}}\times C_{p^{n_2}}$ is isomorphic to $G_{m,n_1,n_2}^{s_1,s_2}$ for some $s_1$ and $s_2$. 
However different five tuples defines isomorphic groups of this type and the group given by the above presentation might not satisfy one of the two conditions $|G'|=p^m$ or $G/G'\cong C_{p^{n_1}}\times C_{p^{n_2}}$. 
To avoid this we only use five tuples satisfying some conditions. 
This is the goal of the following proposition.

\begin{proposition}\label{Clasificacion}
Let $p$ be a prime integer and let $\mathcal{A}$ be the set formed by the tuples of integers $(m,n_1,n_2,s_1,s_2)$ satisfying the following conditions (see Figure~\ref{AdmissiblePoints}):
\begin{itemize}
	\item[(A1)] $0<m$, $0\le s_1,s_2\le m \le n_2\le n_1$,
	\item[(A2)] $n_1-s_1\ge n_2-s_2$,
	\item[(A3)] if $p=2$ then $s_1<n_1$.
	\item[(A4)] either $s_1\ge s_2$ or $p=2$ and $n_1=n_2=m=s_2=s_1+1$.
\end{itemize}
Then every finite 2-generated $p$-group of class 2 is isomorphic to $G_{m,n_1,n_2}^{s_1,s_2}$ for a unique element $(m,n_1,n_2,s_1,s_2) \in \mathcal{A}$.

More precisely, if $G$ is a finite $p$-group of class 2 then the following conditions are equivalente:
\begin{enumerate}
\item $G\cong G_{m,n_1,n_2}^{s_1,s_2}$ with $(m,n_1,n_2,s_1,s_2)\in \mathcal{A}$.
\item $|G'|=p^m$, $G/G'\cong C_{p^{n_1}}\times C_{p^{n_2}}$ with $0<n_2\le n_1$ and $(p^{m+n_1-s_1},p^{m+n_2-s_2})$ is the maximum with respect to the lexicographical order in the following set:
	$$\{(|g_1|,|g_2|) : G/G'=\GEN{\overline{g_1}}\times \GEN{\overline{g_2}}, \text{ for } g_i \in G, 
	\text{ with } |\overline{g_i}|=p^{n_i}\; (i=1,2)\}.$$
\end{enumerate}
\end{proposition}

\begin{proof}
We fix $G$ a finite $p$-group of class 2.
As $G'$ is central in $G$, for every $x,y,z\in G$ and every positive integer $n$ we have 
$$[xy,z]=[x,z][y,z], [x,yz]=[x,y][x,z] \qand  (xy)^n = [y,x]^{\frac{n(n-1)}{2}} x^ny^n.$$
We will use these formulas without specific mention. 
Moreover, by the Burnside Basis Theorem \cite[5.3.2]{Robinson1982}  condition (2) implies that $G$ is 2-generated.
Clearly, if condition (1) holds then $G$ is 2-generated and hence, in the remainder of the proof, we assume that $G$ is $2$-generated. 
The first statement of the proposition is an obvious consequence of the second, so we only have to prove that conditions (1) and (2) are equivalent.

We first prove that (2) implies (1). So we suppose that $|G'|=p^m$, $G/G'\cong C_{p^{n_1}} \times C_{p^{n_2}}$ with $n_1\ge n_2$, and $(p^{o_1},p^{o_2})=(p^{m+n_1-s_1},p^{m+n_2-s_2})$ is the maximum with respect to the lexicographical order in the following set $\{(|g_1|,|g_2|) : (g_1,g_2)\in \G\}$, where 
	\begin{equation}\label{GDef}
	\G=\{(g_1,g_2) \in G\times G: G/G'=\GEN{\overline{g_1}}\times \GEN{\overline{g_2}} \text{ and } |\overline{g_i}|=p^{n_i}\; (i=1,2)\}.
	\end{equation}
Observe that since $G$ is not abelian and $G'$ is central we have $m>0$ and $G/G'$ is $2$-generated but not cyclic. 

We will prove that $(m,n_1,n_2,s_1,s_2)\in \mathcal{A}$ and $G\cong G_{m,n_1,n_2}^{s_1,s_2}$. 
First of all, if $(g_1,g_2)\in \G$ then $[g_2,g_1]^{p^{n_2}}=[g_2^{p^{n_2}},g_1]=1$ and  $G'=\GEN{[g_2,g_1]}$. Thus $m\le n_2\le n_1$.

We will use several times that if $(g_1,g_2)\in \G$ and $t$ is an integer then 
\begin{itemize}
	\item $(g_1g_2^t,g_2)\in \G$, 
	\item if $p\nmid t$ then $(g_1^t,g_2)\in\G$ and $(g_1,g_2^t)\in \G$,
	\item if $p^{n_1-n_2}\mid t$ then $(g_1,g_1^tg_2)\in \G$ and 
	\item if $n_1=n_2$ then $(g_2,g_1)\in \G$.
\end{itemize}

Let 
$$\B=\{(b_1,b_2)\in \G : |b_i|=p^{o_i}\; (i=1,2)\}.$$
Fix $(b_1,b_2)\in \B$ and set $a=[b_2,b_1]$. 
Then $b_i^{p^{n_i}}=a^{t_i}$ for a unique integer $t_i$ with $1\le t_i\le p^m$.
Then $|b_i|=p^{m+n_i-v_p(t_i)}$ and therefore $s_i=v_p(t_i)$.
Therefore $0\le s_i \le m$. 
This proves (A1).

As $(b_1b_2,b_2)\in \G$, by the maximality of $(o_1,o_2)$ we have
	$$1=(b_1b_2)^{p^{o_1}} = a^{\frac{p^{o_1}(p^{o_1}-1)}{2}}b_j^{p^{o_1}}.$$
We claim that $a^{\frac{p^{o_1}(p^{o_1}-1)}{2}}=1$.
Otherwise, as, $m\le n_2\le n_1\le o_1$, necessarily $p=2$ and $o_1=m$. Thus $n_1=n_2=m$ and hence $(b_2,b_1)\in \G$. 
Using the maximality of $(o_1,o_2)$ we deduce that $o_1\ge o_2$.
Hence $b_2^{p^{o_1}}=1$ and so $a^{\frac{p^{o_1}(p^{o_1}-1)}{2}}=1$, as desired. 
We conclude that $n_1-s_1=o_1-m\ge o_2-m=n_2-s_2$. 
Moreover, if $p=2$ then $m+n_1-s_1=o_1>m$, i.e. $s_1<n_1$.
This proves (A2) and (A3).

In order to prove (A4) we assume that $s_1<s_2$. 
We have to show that $p=2$, $n_1=n_2=m=s_2$ and $s_1=m-1$.
We have $o_2+n_1-n_2=m+n_1-s_2<m+n_1-s_1=o_1$ and hence $b_1^{p^{o_2+n_1-n_2}}\ne 1$.
Moreover, $(b_1,b_1^{p^{n_1-n_2}} b_2)\in \G$ and hence, from the maximality of $(o_1,o_2)$ we have 
$$1=(b_1^{p^{n_1-n_2}}b_2)^{p^{o_2}} = 
a^{p^{n_1-n_2}\frac{p^{o_2}(p^{o_2}-1)}{2}} 
b_1^{p^{o_2+n_1-n_2}}.$$
Thus $a^{p^{n_1-n_2}\frac{p^{o_2}(p^{o_2}-1)}{2}}\ne 1$.
By (A1), we have $m\le n_2\le o_2$ and hence $p=2$ and $n_1=n_2=o_2=m$. 
Therefore $s_2=m$ and $a^{2^{m-1}(2^{m}-1)}=b_1^{2^{o_2}}=a^{2^{m-1}}$. 
Thus $o_1=o_2+1=m+1$ and therefore $s_1=m-1$. 
This finishes the proof of (A1)-(A4). 

We now prove that we may take $(b_1,b_2)\in \B$ so that $b_i^{p^{n_i}}=[b_2,b_1]^{p^{s_i}}$ for $i=1,2$. 
Indeed, first write $t_i=u_ip^{s_i}$ and take $c_i=b_i^{u_i}$. 
As $s_i=v_p(t_i)$ we have $p\nmid u_i$ and hence $(c_1,c_2)\in \B$ and 
$c_i^{p^{n_i}} = a^{p^{s_i}}$. 
However, now $[c_2,c_1]=a^{u_1u_2}$. Let $q$ be an integer such that $qu_1u_2\equiv 1 \mod p^m$ and take $B_i=c_i^q$ for $i=1,2$. Then 
    $$[B_2,B_1]=[c_2,c_1]^{q^2} = a^{q^2u_1u_2}=a^q$$
and 
	$$B_i^{p^{n_i}} = a^{qp^{s_i}}=[B_2,B_1]^{p^{s_i}}.$$
So replacing $(b_1,b_2)$ by $(B_1,B_2)$ we have the desired conclusion. 

Therefore $\mathcal{B}$ contains an element $(b_1,b_2)$ satisfying the relations in the presentation of $G_{m,n_1,n_2}^{s_1,s_2}$. Thus $G$ is an epimorphic image of $G_{m,n_1,n_2}^{s_1,s_2}$. As $|G|=p^{m+n_1+n_2} \le |G_{m,n_1,n_2}^{s_1,s_2}|$ we deduce that $G\cong G_{m,n_1,n_2}^{s_1,s_2}$. This finishes the proof of (2) implies (1).

To prove (1) implies (2) it is enough to show that if $G=G_{m,n_1,n_2}^{s_1,s_2}$ then $G/G'\cong C_{p^{n_1}}\times C_{p^{n_2}}$, $|G'|=p^m$, $|b_i|=p^{o_i}$ with $o_i=m+n_i-s_i$ for $i=1,2$ and $(p^{o_1},p^{o_2})$ is the maximum with respect of the lexicographical order in the set $\{(|c_1|,|c_2|):(c_1,c_2)\in\G\}$ where $\G$ is as in \eqref{GDef}. 
The first is clear because the presentation of $G$ yields the following presentation of its abelianizer: 
	$$G/G'\cong \GEN{\overline{b_1},\overline{b_2}:[\overline{b_2},\overline{b_1}]=1, \overline{b_i}^{p^{n_i}}=1 (i=1,2)} \cong C_{p^{n_1}}\times C_{p^{n_2}}.$$

Let $A=\GEN{a}$ be a cyclic group of order $p^m$ and consider the map $f:G/G'\rightarrow A$ given by 
$$f\left(\overline{b_1}^{x_1}\overline{b_2}^{x_2},\overline{b_1}^{y_1}\overline{b_2}^{y_2}\right) = 
a^{x_2y_1}.$$
A straightforward calculation shows that $f$ is a well defined $2$-cocycle and that central extension of $G/G'$ by $A$ associated to $f$ is isomorphic to $G_{m,n_1,n_2}^{s_1,s_2}$. Thus $|G|=p^{m+n_1+n_2}$ and hence $|G'|=p^m$. Now it is clear that $|b_i|=p^{o_i}$. 

It remains to prove that $(p^{o_1},p^{o_2})$ is greater or equal than $(|c_1|,|c_2|)$ with respect to the lexicographical order for every $(c_1,c_2)\in \G$. By means of contradiction let $(c_1,c_2)\in \G$ with $|c_l|>|b_l|$ for some $l\in \{1,2\}$ and if $l=2$ then $|c_1|=|b_1|$.
Clearly $|b_i|=p^{o_i}$ with $o_i=m+n_i-s_i$.
Set $a=[b_2,b_1]$ and fix integers $0\le y_i<p^m$, $0\le x_{i1}<p^{n_1}$ and $0\le x_{i2}<p^{n_2}$ with 
	$$c_i=a^{y_i}b_1^{x_{i1}}b_2^{x_{i2}} \quad (i=1,2).$$
Observe that $p^{n_1-n_2}\mid x_{21}$ because $|\overline{c_2}|=p^{n_2}$.
Thus $v_p(x_{21})\ge n_1-n_2$.
Moreover, as $|c_l|>|b_l|=p^{o_l}$, we have
	\begin{equation}\label{cs}
	1\ne c_l^{p^{o_l}} = a^{y_lp^{o_l}} \cdot 
	a^{x_{l1}x_{l2} \frac{p^{o_l} (p^{o_l}-1)}{2}}\cdot b_1^{x_{l1}p^{o_l}}b_2^{x_{l2}p^{o_l}}.
	\end{equation} 
By (A1), we have $s_l\le n_l$ and hence $m\le o_l$. Thus $a^{y_lp^{o_l}}=1$. 
Moreover, $b_l^{x_{ll}p^{o_l}}=b_2^{x_{l2}p^{o_1}}=1$ because $o_1=m+n_1-s_1\ge m+n_2-s_2=o_2$. Thus $b_2^{x_{l2}p^{o_l}}=1$. 
Combining this with \eqref{cs}, we get
 \begin{equation}\label{ab}
 a^{-x_{l1}x_{l2} \frac{p^{o_l} (p^{o_l}-1)}{2}}\ne b_1^{x_{l1}p^{o_l}}.
 \end{equation}

Suppose that $b_1^{x_{l1}p^{o_l}}\ne 1$. Then $l=2$ and $m+n_1-s_2 \le o_2+v_p(x_{21}) < o_1 = m+n_1-s_1$. Therefore $s_1<s_2$ and hence the second alternative of (A4) holds, i.e. that $p=2$, $m=n_1=n_2=s_2$ and $o_1=m+1$ and hence $v_2(x_{21})=0$. 
This proves the following
\begin{equation}\label{Bs}
b_1^{x_{l1}p^{o_l}} = \begin{cases} a^{2^{m-1}}, & 
\text{if } l=p=2, m=n_1=n_2=s_2=s_1+1 \text{ and } 2\nmid x_{l1};
\\
1, & \text{otherwise}.\end{cases}
\end{equation}
On the other hand
\begin{equation}\label{As}
a^{x_{l1}x_{l2} \frac{p^{o_l} (p^{o_l}-1)}{2}}=
\begin{cases}
a^{2^{m-1}} & \text{ if } p=2, n_l=s_l \text{ and } 2\nmid x_{l1}x_{l2}; \\
1, & \text{otherwise}.
\end{cases}
\end{equation}
Therefore, if $p\ne 2$ or $n_l\ne s_l$, then $c_l^{p^{o_l}}=1$, a contradiction, with \eqref{cs}.
Hence $p=2$ and $n_l=s_l$, i.e. $o_l=m$. 

We claim that $s_1<s_2$.
This follows from \eqref{Bs} if $b_1^{x_{l1}p^{o_l}}\ne 1$. 
Otherwise $a^{x_{l1}x_{l2} \frac{2^{o_l} (2^{o_l}-1)}{2}}\ne 1$. 
Then, by \eqref{As}, $p=2$ and $2\nmid x_{l1}x_{l2}$ and $n_l=s_l$. 
As $s_1<n_1$, by (A3), necessarily $l=2$ and as $0\le n_1-n_2\mid v_2(x_{l1})=0$, it follows that $s_1<n_1=n_2=s_2$, as desired.

Thus, by (A4), $p=2$ and $m=n_1=n_2=s_2=s_1+1$. 
Combining this with \eqref{ab}, \eqref{Bs} and \eqref{As} we conclude that $l=2$ and
	$$a^{2^{m-1}x_{21}x_{22}} = a^{-x_{21}x_{22} \frac{p^{o_2} (p^{o_2}-1)}{2}}\ne b_1^{x_{21}p^{o_2}}b_2^{x_{22}p^{o_2}}=a^{2^{m-1}x_{21}}.$$
Thus $x_{21}$ is odd and $x_{22}$ is even. 
As $(c_1,c_2)\in \G$, necessarily $x_{12}$ is odd. 
By assumption $|c_1|=|b_1|=2^{m+1}$. 
Therefore 
	$$1\ne c_1^{2^m} = a^{x_{11}x_{12} 2^{m-1}(2^m-1)+2^{m-1}x_{11}}=a^{x_{12}2^m}=1$$
which is the desired contradiction. This finishes the proof of the proposition.
\end{proof}

\begin{corollary}\label{Exponente}
If $(m,n_1,n_2,s_1,s_2)\in \mathcal{A}$ then the exponent of $G_{m,n_1,n_2}^{s_1,s_2}$ is $p^{m+n_1-s_1}$.
\end{corollary}

\begin{proof}
As $|b_1|=p^{m+n_1-s_1}$ we have to show that if $x\in G_{m,n_1,n_2}^{s_1,s_2}$ then $x^{p^{m+n_1-s_1}}=1$.
Let $a=[b_2,b_1]$. 
Then $x=a^ib_1^jb_2^k$ for integers $0\le i<p^m, 0\le j<p^{n_1}$ and $0\le k < p^{n_2}$. 
Therefore
\begin{eqnarray*}
x^{p^{m+n_1-s_1}}&=&a^{ip^{m+n_1-s_1}+jk\frac{p^{m+n_1-s_1}(p^{m+n_1-s_1}-1)}{2}}
b_1^{jp^{m+n_1-s_1}} b_2^{kp^{m+n_1-s_1}} \\
&=& a^{jk\frac{p^{m+n_1-s_1}(p^{m+n_1-s_1}-1)}{2}}=1	
\end{eqnarray*}
because, by (A2) and Proposition~\ref{Clasificacion}, we have $|b_2|=p^{m+n_2-s_2}\le p^{m+n_1-s_1} = |b_1|$, and, by (A3), either $s_1<n_1$ or $p\ne 2$.
\end{proof}

We are ready for the proof of the main result. 

\medskip
\begin{proofof}\emph{Theorem}~\ref{Main}.
As the result is known if $G$ is abelian we may assume that $G$ and $H$ are not abelian. From the assumption $\F_p G\cong \F_p H$ it follows that $|G|=|H|$, $G/G'\cong H/H'$ and $\Exp(G)=\Exp(H)$ (see \cite{Ward, Passman1965, Sandling1984, Kulshammer1982, Sandling1996}. Moreover, as $G$ is 2-generated and of class two then so is $H$. Indeed, the latter is proved in \cite{BaginskiKonovalov} and the former is a consequence of Propositions~1.14 and 1.15 in \cite{Sehgal1978} combined with the Burnside Basis Theorem \cite[5.3.2]{Robinson1982}.

By Proposition~\ref{Clasificacion}, there are $(m_1,n_1,n_2,s_1,s_2),(m'_1,n'_1,n'_2,s'_1,s'_2)\in \mathcal{A}$ such that $G\cong G_{m_1,n_1,n_2}^{s_1,s_2}$ and $H\cong G_{m'_2,n'_2,n'_2}^{s'_1,s'_2}$. 
Moreover, $p^{m_1+n_1+n_2}=|G|=|H|=p^{m'_1+n'_1+n'_2}$,
$C_{p^{n_1}}\times C_{p^{n_2}} \cong G/G' \cong H/H' \cong C_{p^{n'_1}}\times C_{p^{n'_2}}$ and as $n_1\ge n_2$ and $n'_1\ge n'_2$ it follows that $n_1=n_1'$, $n_2=n'_2$ and $m=m'$. 
Furthermore, by Corollary~\ref{Exponente} we have $p^{m_1+n_1-s_1}=\Exp(G)=\Exp(H)=p^{m_1+n_1-s'_1}$, so that $s_1=s'_1$. 

It remains to prove that $s_2=s'_2$.
To this end we first compute the dimension subgroups $M_n(G)$ of $G$. 
Recall that if $\gamma_i(G)$ denotes the $i$-th term of the minimal central series of $G$ and $n$ is a positive integer then $M_n(G)=\prod_{ip^j\ge n}\gamma_i(G)^{p^j}$. (Here, $X^u$ denotes the group generated by the elements of the form $x^u$ with $x\in X$, for $X\subseteq G$ and $u$ a positive integer.) 
As $G$ is of class $2$, we have 
    \begin{equation}\label{Dimension}
     M_n(G)=\begin{cases}
        G, & \text{if } n=1; \\
        {G'}^{p^k} G^{p^{k+1}} = \GEN{a^{p^k},b_1^{p^{k+1}},b_2^{p^{k+1}}}, & \text{if } p^k < n \le 2p^k; \\
        G^{p^{k+1}} = \GEN{a^{p^{k+1}},b_1^{p^{k+1}},b_2^{p^{k+1}}}, & \text{if } 2p^k < n \le p^{k+1}.
          \end{cases}
    \end{equation}
In particular, if $M_n(G)/M_{n+1}(G)\ne 1$ then $n$ is either $p^k$ or $2p^k$ for some non-negative integer $k$.  

Each quotient $M_i(G)/M_{i+1}(G)$ is an elementary abelian $p$-group and, by \cite[Theorems~3.6 and 5.5]{Jennings1941}, 
$[M_i(G):M_{i+1}(G)]=[M_i(H):M_{i+1}(H)]$ for every $i\ge 1$. Furthermore, if $\Delta(G)$ denotes the augmentation ideal of $\F_pG$ and $x_1,\dots,x_r$ is a minimal generating set of $M_i(G)/M_{i+1}(G)$ then $(1+x_1)+\Delta(G)^{i+1},\dots,(1+x_r)+\Delta(G)^{i+1}$ is a basis of $\Delta(G)^i/\Delta(G)^{i+1}$ as vector space over $\F_p$.

For $k\ge 0$ define $d_k$ by the following equality:
    $$p^{d_k} = [M_{2p^k}(G):M_{2p^k+1}(G)] = [M_{2p^k}(H):M_{2p^k+1}(H)].$$
By \eqref{Dimension} we have 
	$$p^{d_k} = \begin{cases}
	\left[\GEN{a^{p^k},b_1^{p^{k+1}},b_2^{p^{k+1}}}:
	\GEN{a^{p^{k+1}},b_1^{p^{k+1}},b_2^{p^{k+1}}}\right], & \text{if } p\ne 2; \\\\
	\left[\GEN{a^{p^{k}},b_1^{p^{k+1}},b_2^{p^{k+1}}}:
	\GEN{a^{p^{k+1}},b_1^{p^{k+2}},b_2^{p^{k+2}}}\right], & \text{if } p=2; \\
	\end{cases}$$    

Suppose first that $p\ne 2$ and let $u = \min\{i\ge 1 : d_i=0\}$.
We claim that $u=s_2$. 
First observe that $a^{p^{s_2}}=b_2^{p^{n_2}}\in \GEN{b_2^{p^{s_2+1}}}$ for otherwise $s_2\ge n_2$, and as $s_2\le m \le n_2$, we have $a^{p^{s_2}}=1$.
Therefore $u\le s_2$. 
To prove equality we take $i<s_2$ and we have to show that $a^{p^i}\not\in M_{2p^i+1}(G)$. Otherwise, as $\GEN{b_1,b_2}\cap \GEN{a}=\GEN{a^{p^{s_2}}}$, we have $a^{p^i} \in \GEN{a^{p^{i+1}},a^{p^{s_2}}}=\GEN{a^{p^{i+1}}}$. 
This only holds if $i\ge m$ which is not the case because $i<s_2\le m$.
This proves the result for the case where $p\ne 2$ because the same argument shows that $u$ has to be also equal to $s_2'$.

In the remainder of the proof we assume that $p=2$ and let $u = \min\{i\ge 1 : d_i<3\}$.
We claim that $u=\min(n_2-1,s_1,s_2)$. 
Firstly, as $p=2$, we have $s_1<n_1$, by (A3), and therefore $a^{p^{s_1}}=b_1^{p^{n_1}} \in \GEN{b_1^{p^{s_1+1}}}$. Thus $d_{s_1}<3$ and hence $u\le s_1$. 
Secondly, $a^{p^{s_2}}\in \GEN{b_2^{p^{n_2+1}}}$, because if $s_2<n_2$ then $a^{p^{s_2}}=b_2^{n_2}\in \GEN{b_2^{p^{s_2+1}}}$, and otherwise $s_2=m=n_2$ so that $a^{p^{s_2}}=1$. Therefore $d_{s_2}<3$ and hence $u\le s_2$.
Finally, if $s_2<n_2$ then $n_2-1\ge m$ and hence $a^{p^{n_2-1}}=1$. Otherwise, i.e. if $s_2\ge n_2$ then $s_2=m=n_2$, so that $b_2^{n_2}=a^{p^{s_2}}=1$. Therefore $d_{n_2-1}<3$ and hence $u\le n_2-1$.
This proves that $u\le \min(s_1,s_2,n_2-1)$. 
To prove that the equality holds we have to show that if $0\le k < \min(n_2-1,s_1,s_2)$ then $d_k=3$.
Indeed, let $1\le k < \min(n_2-1,s_1,s_2)$ and let $x,y,z$ be integers such that $a^{2^kx}b_1^{2^{k+1}y}b_2^{2^{k+1}z}\in \GEN{a^{2^{k+1}},b_1^{2^{k+2}},b_1^{2^{k+2}}}$. 
Then 
    $$a^{2^kx}b_1^{2^{k+1}y}b_2^{2^{k+1}z}=a^{2^{k+1}x_1}b_1^{2^{k+2}y_1}b_2^{2^{k+2}z_1}$$
for some integers $x_1,y_1$ and $z_1$. 
Hence 
    $$b_1^{2^{k+1}(y-2y_1)} b_2^{2^{k+1}(z-2z_1)}\in \GEN{a}$$
and therefore $2^{n_1}\mid 2^{k+1}(y-2y_1)$ and $2^{n_2}\mid 2^{k+1}(z-2z_1)$. As $k+1<n_2\le n_1$ it follows that $y$ and $z$ are even. Furthermore 
	$$a^{2^k(2x_1-x)} = b_1^{2^{n_1}\alpha} b_2^{2^{n_2}\beta} = a^{\alpha s_1+\beta s_2}$$
for some integers $\alpha$ and $\beta$. 
As $k<\min(s_1,s_2)\le m$ it follows that $x$ is even too.
This shows that the elements of the form $a^{2^kx}b_1^{2^{k+1}y}b_2^{2^{k+1}z}$ with $0\le x,y,z\le 1$ are all different and therefore $2^{d_k}=|M_{2^k}(G)/M_{2^{k}+1}(G)|=8$, i.e. $d_k=3$, as desired. 
This finishes the proof of the claim.

Of course the same argument shows that $u=\min(n_2-1,s_1,s'_2)$. 

So we have $u=\min(n_2-1,s_1,s_2)=\min(n_2-1,s_1,s'_2)$. 
If $s_2<\min(n_2-1,s_1)$ then $s_2=u=\min(n_2-1,s_1,s_2')=s'_2$.
Similarly, if $s'_2<\min(n_2-1,s_1)$ then $s_2=s'_2$. 
Thus we may assume that $\min(n_2-1,s_1)\le s_2< s_2'$. 
We claim that $n_2-1\le s_1$. Otherwise $s_1<s_2'$ and hence by condition (A4), applied to $(m,n_1,n_2,s_1,s'_2)$, we deduce that $n_2=s_1+1$, so that $n_2-1\le s_1$, as desired. Therefore $u=n_2-1\le s_2<s_2'\le m \le n_2$, by condition (A1), so that
	$$s_2'=m=n_2 \qand u=s_2=m-1\le s_1\le m.$$

We now use an argument from \cite{Passman1965} (see also \cite{Jennings1941}) which eventually will yield a contradiction with the previous equalities. 
Let $\Delta=\Delta(G)$ and for each $i\ge 1$ consider the natural map $f_i:\Delta^{2^{i-1}}/\Delta^{2^{i-1}+1} \rightarrow \Delta^{2^i}/\Delta^{2^i+1}$ given by $f_i(x)=x^2$. 
One defines similarly maps $g_i:\Delta_1^{2^{i-1}}/\Delta_1^{2^{i-1}+1} \rightarrow \Delta_1^{2^i}/\Delta_1^{2^i+1}$ where $\Delta_1=\Delta(H)$. 
The isomorphism $\F_2G\cong \F_2H$ implies that the sets 
	$$X_{i,G}=\{x\in \Delta/\Delta^2 : f_i f_{i-1} \dots f_1(x)=0\}$$
and
	$$X_{i,H}=\{x\in \Delta/\Delta^2 : g_i g_{i-1} \dots g_1(x)=0\}$$ 
have the same cardinality.

Observe that $\{1+b_1+\Delta^2,1+b_2+\Delta^2\}$ is a basis of $\Delta/\Delta^2$ and  $\{1+a^{2^{i-1}}+\Delta^{2^{i}+1},1+b_1^{2^{i}}+\Delta^{2^{i}+1},1+b_2^{2^{i}}+\Delta^{2^{i}+1}\}$ generates $\Delta^{2^{i}}/\Delta^{2^{i}+1}$, as vector space over $\F_2$, for every $i\ge 1$. 
Using the formula 
    $$(x-1)(y-1) \equiv (y-1)(x-1)+((x,y)-1) \mod \Delta^{i+j+1}$$ 
whenever $x-1\in \Delta^i$ and $y-1\in \Delta^j$,
we deduce that if $\alpha_1,\alpha_2\in \F_2$ then 
    $$(\alpha_1(1+b_1)+\alpha_2(1+b_2))^2 \equiv \alpha_1 (1+b_1)^2 + \alpha_2 (1+b_2)^2 + \alpha_1\alpha_2(1+a) \mod \Delta^3.$$
Thus $1+a\in \Delta^2$ and hence $1+a^4 = 1+(b_2^2,b_1^2)\in \Delta^8$. 
Therefore 
    $$(\alpha_1(1+b_1)+\alpha_2(1+b_2))^4 \equiv \alpha_1 (1+b_1^4) + \alpha_2 (1+b_2^4) + \alpha_1\alpha_2(1+a^2) \mod \Delta^5.$$
More generally, for every $i\ge 1$ we have 
    $$(\alpha_1(1+b_1)+\alpha_2(1+b_2))^{2^i} \equiv 
    \alpha_1 (1+b_1^{2^i}) + \alpha_2 (1+b_2^{2^i}) + \alpha_1\alpha_2(1+a^{2^{i-1}}) \mod \Delta^{2^i+1}.$$
    
Suppose first that $s_1=m-1$. As $s_1<s'_2$, applying condition (A4) to $(m,n_1,n_2,s_1,s'_2)$ it follows that $m=n_1$. 
Then 
    \begin{eqnarray*}
     &&f_mf_{m-1}\cdots f_1(\alpha_1(1+b_1)+\alpha_2(1+b_2)) \\
     &\quad &= \alpha_1 (1+a^{2^{s_1}}) + \alpha_2 (1+a^{2^{s_2}}) + \alpha_1\alpha_2 (1+a^{2^{m-1}})  \\
     &\quad &=  (\alpha_1 + \alpha_2+ \alpha_1\alpha_2)(1+a^{2^{m-1}})
    \end{eqnarray*}
 and 
    \begin{eqnarray*}
     &&g_mg_{m-1}\cdots g_1(\alpha_1(1+b_1)+\alpha_2(1+b_2)) \\
     &\quad &= \alpha_1 (1+a^{2^{s_1}}) + \alpha_2 (1+a^{2^{s'_2}}) +\alpha_1\alpha_2 (1+a^{2^{m-1}})  \\
     &\quad &=  \alpha_1 (1+\alpha_2)(1+a^{2^{m-1}})
    \end{eqnarray*}
Therefore $X_{m,G}=\{(0,0)\}$ while $X_{m,H}=\{(0,0),(0,1),(1,1)\}$, a contradiction.

Suppose now that $s_1=m$. Then $m<n_1$ by condition (A3). Then 
    \begin{eqnarray*}
     &&f_mf_{m-1}\cdots f_1(\alpha_1(1+b_1)+\alpha_2(1+b_2)) \\
     &\quad &= \alpha_1 (1+b_1^{2^m}) + \alpha_2 (1+a^{2^{s_2}}) + \alpha_1\alpha_2 (1+a^{2^{m-1}})  \\
     &\quad &=  \alpha_1 (1+b_1^{2^m}) + \alpha_2(1+\alpha_1)(1+a^{2^{m-1}})
    \end{eqnarray*}
 and 
    \begin{eqnarray*}
     &&g_mg_{m-1}\cdots g_1(\alpha_1(1+b_1)+\alpha_2(1+b_2)) \\
     &\quad &= \alpha_1 (1+b_1^{2^m}) + \alpha_2 (1+a^{2^{s'_2}}) + \alpha_1\alpha_2 (1+a^{2^{m-1}})  \\
     &\quad &=  \alpha_1 (1+b_1^{2^m}) + \alpha_1\alpha_2(1+a^{2^{m-1}})
    \end{eqnarray*}
Hence $X_{m,G}=\{(0,0)\}$ while $X_{m,H}=\{(0,0),(0,1)\}$, again a contradiction.
This finishes the proof of Theorem~\ref{Main}.
\end{proofof}

\bibliographystyle{amsalpha}
\bibliography{ReferencesMSC}

\end{document}